\documentclass[11pt]{amsart}
\usepackage[foot]{amsaddr}
\usepackage{amsmath,amssymb,bm, graphicx,float}

\begin{document}
\newtheorem{theorem}{Theorem}[section]
\newtheorem{lemma}[theorem]{Lemma}
\newtheorem{corollary}[theorem]{Corollary}
\newtheorem{prop}[theorem]{Proposition}
\newtheorem{definition}[theorem]{Definition}
\newtheorem{remark}[theorem]{Remark}

 \def\ad#1{\begin{aligned}#1\end{aligned}}  \def\b#1{{\mathbf #1}} \def\hb#1{\hat{\bf #1}}
\def\a#1{\begin{align*}#1\end{align*}} \def\an#1{\begin{align}#1\end{align}}
\def\e#1{\begin{equation}#1\end{equation}} \def\t#1{\hbox{\rm{#1}}}
\def\dt#1{\left|\begin{matrix}#1\end{matrix}\right|}\def\d{\t{div}\,}
\def\p#1{\begin{pmatrix}#1\end{pmatrix}}  
  \numberwithin{equation}{section} 
\def\boxit#1{\vbox{\hrule height1pt \hbox{\vrule width1pt\kern1pt
     #1\kern1pt\vrule width1pt}\hrule height1pt }}  
 
\title  [Superconvergent quadratic tetrahedral element]
   {Superconvergent quadriatic finite element on uniform tetrahedral meshes}

\author{ Yunqing Huang }
\address{ School of Mathematics and Computational Science, Xiangtan University, Xiangtan, Hunan, 411105,
   China.}
\email{huangyq@xtu.edu.cn}
\thanks{Yunqing Huang was supported in part by National Natural Science Foundation of China Project (11971410) and China’s National Key R\&D Programs (2020YFA0713500).}

\author { Shangyou Zhang }
\address{Department of Mathematical  Sciences, University of Delaware, Newark, DE 19716, USA}
\email{szhang@udel.edu}

\date{}

\subjclass{ 65N15, 65N30}

\keywords{conforming finite element, superconvergence, 
  quadratic finite element, Poisson's equation, tetrahedral mesh.}

\begin{abstract} By a direct computation,  we show that the $P_2$ interpolation of a $P_3$
  function is also a local $H^1$-projection on uniform tetrahedral meshes, i.e.,
  the difference is $H^1$-orthogonal to the  $P_2$ Lagrange basis function on the
   support patch of tetrahedra of the basis function.
Consequently, we show the $H^1$ and $L^2$ superconvergence of the $P_2$ Lagrange finite element
   on uniform tetrahedral meshes.
Using the standard 20-points Lagrange $P_3$ interpolation,  where the 20 nodes are
  exactly some $P_2$ global basis nodes, we lift the superconvergent $P_2$ finite
   element solution to a quasi-optimal $P_3$ solution on each cube. 
Numerical results confirm the theory. 
\end{abstract}

\maketitle

%%%%%%%%%%%%%%%%%%%%%%%%%%%%%%%%%%%%%%%%%%%%%%%%%%%%%%%%%%%%%%%%%%%%%%%%%%%
\section{Introduction}

We solve the 3D Poisson equation on uniform tetrahedral meshes by the quadratic 
  Lagrange finite element:  
  \an{\label{e1-1} -  \Delta u  &= f \qquad  \hbox{in } \Omega,  \\
       \label{e1-2}     u  &= 0 \qquad \hbox{on } \partial\Omega,  }
  where the polyhedral domain $\Omega$ can be subdivided into uniform 
     (i.e. every tetrahedron is one of the six cut from a cube of same size $h$)
   tetrahedral meshes $\mathcal T_h$,     
     shown in Figure \ref{p3n}.  
The variational form of \eqref{e1-1}--\eqref{e1-2} reads:  
    Find $u\in H_0^1(\Omega) $   such that 
\a{ \ad{  (\nabla u, \nabla v)  &=( f, v) \qquad  \forall  v \in H_0^1(\Omega). }
   }   
The quadratic Lagrange finite element space is defined by
\an{\label{V-h} V_h=\{ v_h\in H_0^1(\Omega) :
   v_h|_{T}\in P_2(T), \ T\in \mathcal T_h\}.  }
The finite element discretization of \eqref{e1-1}--\eqref{e1-2} reads: Find $u\in V_h$
   such that 
\an{ \label{finite}  (\nabla u_h, \nabla v_h)  &=( f, v_h) \qquad  \forall  v_h \in V_h. }

   By the standard finite element theory,  we have the optimal order error estimates that
\a{      |u-u_h|_1 &\le Ch^2 |u|_3, \\ 
          \|u-u_h\|_0 &\le Ch^3 |u|_3.  }
Its $H^1$ superconvergence is proved in \cite{Brandts} that
\a{              |I_h u-u_h|_1 &\le Ch^3 |u|_4, }
  where $I_h$ is the $P_2$ Lagrange interpolation operator.
Its $L^\infty$ and $L^2$ superconvergence is proved in \cite{Liu} that
\a{  \|I_h u-u_h\|_{0,\infty}  &\le Ch^4 |\ln h|^2 C(u), \\   
        \|I_h u-u_h\|_0        &\le Ch^4 C(u),   }
  where $C(u)$ is some very high order norm of $u$.
Based on a weak orthogonality that
\an{\label{w-o} |(\nabla(u-I_h u), \Pi_h^1 v)| \le C h^4 \|u\|_4 \|v\|_2,
     \quad u,v\in H^1_0(\Omega),}  
   where $\Pi_h^1$ the $P_1$ Lagrange interpolation operator,
  the $L^2$ superconvergence is also proved in \cite{Yang}.
  
 \begin{figure}[H] 
  \setlength{\unitlength}{0.8pt}

 \begin{center}\begin{picture}(300.,200 )(-80.,-20.)

 \put(0,0){\begin{picture}(120.,120)(  0.,  0.)   
 
    \multiput(0,20)(4,0){20}{\circle*{1}} \multiput(80,20)(4,-2){10}{\circle*{1}} 
     \multiput(80,20)(0,4){20}{\circle*{1}}  \multiput(80,20)(-3,3){26}{\circle*{1}} 
      \multiput(80,20)(2,3){20}{\circle*{1}} \multiput(80,20)(-4,-2){10}{\circle*{1}}
      \multiput(80,20)(-2,3){20}{\circle*{1}}
      
      \put(60,50){\circle*{4}} \multiput(-40,20)(40,0){4}{\circle*{4}}
      \multiput(20,10)(40,0){3}{\circle*{4}}\multiput(80, 0)(40,0){2}{\circle*{4}}
      \put(140,-10){\circle*{4}}  \put(100,50){\circle*{4}} \put(120,40){\circle*{4}}
      \multiput(40,100)(40,0){2}{\circle*{4}}  \put(100,90){\circle*{4}}
        \put(80,140){\circle*{4}}
        \multiput(0,60)(40,0){3}{\circle*{4}}
      \put(0, 20){\line(-1, 0){80}}  \put(120,0){\line(2,-1){40}} \put(80,100){\line(0,1){80}}
      
     \put(0,100){\line(1,0){80}}\put(0,100){\line(0,-1){80}}\put(0,100){\line(2,-1){40}}
     \put(40, 0){\line(-2, 1){40}}\put(40, 0){\line(0, 1){80}}\put(40, 0){\line(1,0){80}}
     \put(120,80){\line(-2, 1){40}}\put(120, 80){\line(0,-1){80}}\put(120, 80){\line(-1,0){80}}
     \put(40,80){\line(2, 1){40}}\put(40, 80){\line(-2,-3){40}}\put(40, 80){\line( 1,-1){80}}
     
       \end{picture} }

 \end{picture}\end{center}

\caption{Using the standard 20 $P_3$ Lagrange nodes, which are exactly 
    some existing $P_2$ finite basis nodes,
   to lift a superconvergent $P_2$ finite element solution to a
   quasi-optimal $P_3$ solution. }
\label{p3n}
\end{figure}
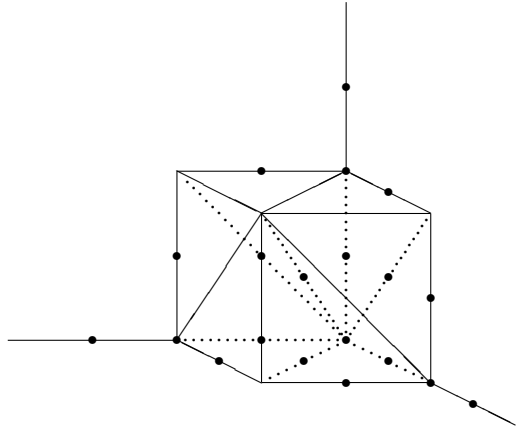 
 
All above works are based on the finite element approximation of the Green's function $G_z$.
But the finite element $H^1$-projection(approximation) of Green's functions
   itself is not computable.
As Green's functions are singular, such error estimates may lead to extra $\ln$-terms or
  may not result in the optimal bound on the right hand side.
We revisit this problem in this work, 
   intending to prove the optimal superconvergent bounds, 
\an{ \label{s1-1}     |I_h u-u_h\|_{1}  &\le Ch^4 |u|_4, \\   
     \label{s1-2}    \|I_h u-u_h\|_0        &\le Ch^4|u|_4, }
where $u$ is the solution of \eqref{e1-1}--\eqref{e1-2}, $u_h$ is the
  solution of \eqref{finite} and $I_h$ is the $P_2$ Lagrange interpolation operator
     on mesh $\mathcal T_h$.

By a direct computation,  we show that the $P_2$ interpolation of a $P_3$
  function is also a local $H^1$-projection on uniform tetrahedral meshes, i.e.,
\a{ (\nabla(u-I_h u), \phi_i)=0 \qquad \t{if } \ u\in P_3(R_i), }
where $\phi_i$ is a global basis of $P_2$ finite element space $V_h$,
  and $R_i=\cup_{\phi_i|_T\ne 0} T$ is a patch of tetrahedra. 
Based on it, we prove the optimal form of \eqref{w-o} above that 
\a{ |(\nabla(u-I_h u), v_h)| \le C h^4 |u|_4 \|v_h|_1,
     \quad u \in H^4(\Omega),\ v\in V_h. }   
Consequently, we show the $H^1$ and $L^2$ superconvergence of the $P_2$ Lagrange finite element
   on uniform tetrahedral meshes, \eqref{s1-1}--\eqref{s1-2}.
   
Using the standard 20-points Lagrange $P_3$ interpolation,  where the 20 nodes are
  exactly some $P_2$ global basis nodes, we lift the superconvergent $P_2$ finite
   element solution to a quasi-optimal $P_3$ solution on each cube. 
Such a lift has almost no computational cost, but it provides quasi-optimal $P_3$ solutions.
Numerical tests are given, showing the superconvergence of $P_2$ finite element solutions and
  the quasi-optimality of lifted $P_3$ solutions.
  
We list some references on 2D $P_2$ superconvergence on uniform triangular meshes,
  \cite{Andreev,Andreev-Lazarov,Bank,Goodsell,Goodsell2,Huang}.

%%%%%%%%%%%%%%%%%%%%%%%%%%%%%%%%%%%%%%%%%%%%%%%%%%%%%%%%%%%%%%%%%%%%%%%%%%%
\section{Local $H^1$-orthogonality of the $P_2$ interpolation}
  
We study such local $H^1$-orthogonality of the $P_2$ interpolation of $P_3$ functions in
  this section.
 
\begin{lemma}\label{l-d} 
Let $p\in P_3(R_1)$,   where $R_1$ is the support patch of tetrahedra of
   mid-cube node basis function $\phi_1$, cf. Figure \ref{3r1}.
    Then the $P_2$ nodal interpolation of $p$, $I_2 p$, is the local $H^1$-projection of $p$,
     i.e., 
\an{\label{d-3-1} (\nabla(p-I_2 p), \nabla \phi_1)_{R_1} = 0.  }  
 \end{lemma} 

\begin{proof}  We shift and isotropic scale the patch  $R_1$ to
  the one  shown in Figure \ref{3r1}.
  
 \begin{figure}[H] 
  \setlength{\unitlength}{1.2pt}

 \begin{center}\begin{picture}(120.,100 )(-10., 0.)

 \put(0,0){\begin{picture}(120.,120)(  0.,  0.)  
  \put(-31,20){$(1,0,0)$} \put(123,75){$(0,1,1)$} 
    \multiput(0,20)(4,0){20}{\circle*{1}} \multiput(80,20)(4,-2){10}{\circle*{1}} 
     \multiput(80,20)(0,4){20}{\circle*{1}}  \multiput(80,20)(-3,3){26}{\circle*{1}} 
      \multiput(80,20)(2,3){20}{\circle*{1}} \multiput(80,20)(-4,-2){10}{\circle*{1}}
      \multiput(80,20)(-2,3){20}{\circle*{1}}
      \put(60,50){\circle*{4}}
      
      \put(-20,6){$T_1$}\put(-10,10){\vector(1,0){28}}
      \put(-30,56){$T_6$}\put(-20,60){\vector(1,0){20}}
      \put(143,36){$T_3$}\put(140,40){\vector(-1,0){20}}
      \put(127,86){$T_4$}\put(125,90){\vector(-1,0){23}}
      
     \put(0,100){\line(1,0){80}}\put(0,100){\line(0,-1){80}}\put(0,100){\line(2,-1){40}}
     \put(40, 0){\line(-2, 1){40}}\put(40, 0){\line(0, 1){80}}\put(40, 0){\line(1,0){80}}
     \put(120,80){\line(-2, 1){40}}\put(120, 80){\line(0,-1){80}}\put(120, 80){\line(-1,0){80}}
     \put(40,80){\line(2, 1){40}}\put(40, 80){\line(-2,-3){40}}\put(40, 80){\line( 1,-1){80}}
     
       \end{picture} }

 \end{picture}\end{center}

\caption{The support patch $R_1$ of tetrahedra of the $P_2$ Lagrange nodal basis $\phi_1$ at
   a mid-cube node. }
\label{3r1}
\end{figure}
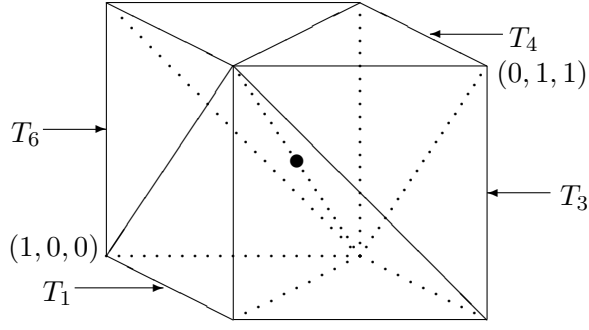

On $R_1$, the only internal basis function $\phi_1$ and its gradient are 
\an{\label{phi-1} 
   & \phi_1=\begin{cases} -4z(x-1) &\t{on} \ T_1, \\
                -4z(y-1)&\t{on} \ T_2, \\
                -4x(y-1)&\t{on} \ T_3, \\
                -4x(z-1)&\t{on} \ T_4, \\
                -4y(z-1)&\t{on} \ T_5, \\
                -4y(x-1)&\t{on} \ T_6, \end{cases} }
and
  \an{\label{phi-1d} &\ad{ \nabla\phi_1|_{T_1}&= \p{-4z \\ 0\\ -4x+4  },
       &   \nabla\phi_1|_{T_2}&= \p{ 0 \\-4z \\ -4y+4 }, \\
      \nabla\phi_1|_{T_3}&=  \p{-4y+4\\ -4x \\0  },
       &   \nabla\phi_1|_{T_4}&=  \p{ -4z + 4 \\ 0 \\ -4x  }, \\
      \nabla\phi_1|_{T_5}&= \p{ 0 \\ -4z+4 \\ -4y  },
       &   \nabla\phi_1|_{T_4}&=  \p{-4y \\ -4x + 4 \\ 0 }. } }

On $R_1$, we need to show three cases that $p=x^3$, $p=x^2y$ and $p=xyz$, 
   due to $xyz$-symmetry of $R_1$.
 For $p=x^3$,  on the six tetrahedra $T_i$,  cf. Figure \ref{3r1}, we have
 \an{\label{1-p} \ad{ p-I_2 p &= \frac 12 x(x-1)(2x-1), \ \t{on} \ T_i, i=1,\dots, 6, \\
    \nabla(p-I_2p) &= 
        \p{\frac 12-3x+3x^2\\0\\0 }, \ \t{on} \ T_i, i=1,\dots, 6. } }
 By \eqref{phi-1},  \eqref{phi-1d} and  \eqref{1-p}, we find the integral on the
   six tetrahedra $T_i$, 
\a{  (\nabla p- \nabla I_2 p, \nabla \phi_1)_{R_1} 
                    &=-\frac 1{30}+0+\frac 1{30}+\frac 1{30}+0-\frac 1{30} =0.  }
                    
For $p=x^2y$,   on $R_1$, we have 
 \an{\label{1-p2} \ad{ p-I_2 p &=\begin{cases}
      y(x-1)(2x-1)/2, & \t{on} \ T_i, i=1,5, 6, \\
      x(y-1)(2x-1)/2, & \t{on} \ T_i, i=2,3,4, \end{cases} \\ 
    \nabla(p-I_2p) &= \begin{cases}
      \p{-\frac 32 y+ 2x y \\-\frac 32 x+\frac 12 + x^2 \\0 }, & \t{on} \ T_i, i=1,5, 6, \\
      \p{ \frac12-2x -\frac 12 y+ 2x y \\-\frac 12 x+ x^2 \\0 }, & \t{on} \ T_i, i=2,3,4. \end{cases} } }
By \eqref{phi-1},  \eqref{phi-1d} and  \eqref{1-p2}, we find the integral on the
   six tetrahedra $T_i$ of $R_1$, 
\a{   (\nabla (p- I_2 p), \nabla \phi_1)_{R_1}  
    =-\frac 1{30}-\frac 1{60}+\frac 1{45}+\frac 1{30}+\frac 1{60}-\frac 1{45}
                     =0.  }
                     
For $p=xyz$,     on $R_1$, we have 
 \an{\label{1-p3} \ad{ p-I_2 p &=\begin{cases}
      z(x-1)(2y-1)/2, & \t{on} \ T_1, \\
                   z(y-1)(2x-1)/2, & \t{on} \ T_2, \\
                   x(y-1)(2z-1)/2, & \t{on} \ T_3, \\
                   x(z-1)(2y-1)/2, & \t{on} \ T_4, \\
                   y(z-1)(2x-1)/2, & \t{on} \ T_5, \\
                   y(x-1)(2z-1)/2, & \t{on} \ T_6, \end{cases} } }
and $\nabla(p-I_2p)=$
\an{\label{1-p3d} \begin{cases}
    \p{-\frac 12 z+ y z  
               \\xz-z 
                \\ -\frac 12 x+\frac 12 -y + x y   }, &T_1, \\
     \p{y z-z
              \\-\frac 12 z+x z 
               \\ -\frac 12 y + \frac 12 - x + x y  }, &T_2, \\
     \p{ - \frac 1 2  y +  \frac 1 2  - z + y z
               \\- \frac 1 2  x + x z
                \\ x y - x }, &T_3, \end{cases} &&  \begin{cases}
    \p{- \frac 1 2  z +  \frac 1 2  - y + y z
            \\ x z - x
             \\ - \frac 1 2  x + x y   }, &T_4, \\
     \p{ y z - y
          \\ - \frac 1 2  z +  \frac 1 2  - x + x z
           \\ - \frac 1 2  y + x y   }, &T_5, \\
     \p{- \frac 1 2  y + y z
             \\- \frac 1 2  x +  \frac 1 2  - z + x z
              \\x y - y   }, &T_6. \end{cases} } 
By \eqref{phi-1},  \eqref{phi-1d},  \eqref{1-p3} and  \eqref{1-p3d}, 
    we find the integral on the six tetrahedra $T_i$ of $R_1$, 
\a{  (\nabla (p- I_2 p), \nabla \phi_1)_{R_1}  
    =-\frac 1{45}-\frac 1{90}+\frac 1{90}+\frac 1{45}+\frac 1{90}-\frac 1{90}
                     =0.  } 
                     
  The proof is complete.
\end{proof}

\begin{lemma} 
Let $p\in P_3(R_1)$,   where $R_1$ is the support patch of tetrahedra of
   mid-square node basis function $\phi_2$, cf. Figure \ref{3r2}.
    Then the $P_2$ nodal interpolation of $p$, $I_2 p$, is the local $H^1$-projection of $p$,
     i.e., 
\an{\label{d-3-2} (\nabla(p-I_2 p), \nabla \phi_2)_{R_2} = 0.  }  
 \end{lemma}

\begin{proof} We shift and isotropic scale the patch  $R_2$ to
  the one  shown in Figure \ref{3r2}.

 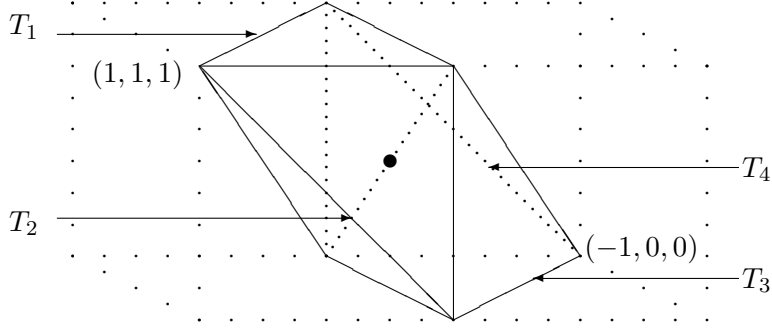
\begin{figure}[H] 
  \setlength{\unitlength}{1.2pt}

 \begin{center}\begin{picture}(220.,100 )(-10., 0)

 \put(0,0){\begin{picture}(220.,120)(  0.,  0.)  
  \put(6,75){$(1,1,1)$} \put(161.3,20){$(-1,0,0)$} 
    \multiput(0,20)(10,0){16}{\circle*{1}} \put(80,20){\line(2,-1){40}}
    
      \multiput(160,20)(0, 10){9}{\circle*{1}} \multiput(160,20)(10, -5){4}{\circle*{1}} 
      \multiput(200,80)(0,-10){9}{\circle*{1}} 
      \multiput(200,80)(-10,0){9}{\circle*{1}} \multiput(200,80)(-10,5){4}{\circle*{1}}
     \multiput(80,20)(0,4){20}{\circle*{1}}  
      \multiput(80,20)(2,3){20}{\circle*{1}} \put(80,20){\line(-2,3){40}}
      \multiput(80,100)(2.5,-2.5){32}{\circle*{1}} 
      \put(100,50){\circle*{4}}
      \put(120,0){\line(2,1){40}}
      \put(120,80){\line(2,-3){40}}
      
      \put(-20,28){$T_2$}\put(-5,32){\vector(1,0){93 }}
      \put(-20,90){$T_1$}\put(-5,90){\vector(1,0){63}}
      \put(211,10){$T_3$}\put(210,13){\vector(-1,0){65}}
      \put(211,45){$T_4$}\put(210,48){\vector(-1,0){79}}
      
     \multiput(0,100)(10,0){16}{\circle*{1}}
     \multiput(0,100)(0,-10){8}{\circle*{1}}
     \multiput(0,100)(10,-5){4}{\circle*{1}}
     \multiput(0,20)(10,-5){4}{\circle*{1}}
     \multiput(40, 0)(10,0){16}{\circle*{1}}
     \multiput(40, 0)( 0,10){8}{\circle*{1}}
      
     \put(120,80){\line(-2, 1){40}}
     \put(120, 80){\line(0,-1){80}}\put(120, 80){\line(-1,0){80}}
     \put(40,80){\line(2, 1){40}} \put(40, 80){\line( 1,-1){80}} 
     
       \end{picture} }

 \end{picture}\end{center}

\caption{The support patch $R_2$ of tetrahedra of the $P_2$ Lagrange nodal basis $\phi_2$ at
   a mid-square node. }
\label{3r2}
\end{figure} 
 
 On the four tetrahedra of $R_2$,  we find
\an{\label{phi-2} \phi_2=\begin{cases}
   x (x - 1)(2 x - 1) /2 &\t{on }\ T_1, 
                    \\x (x - 1)(2 x - 1) /2 &\t{on }\ T_2, 
                    \\x (x + 1) (2 x + 1)/2 &\t{on }\ T_3, 
                    \\x (x + 1) (2 x + 1)/2 &\t{on }\ T_4, \end{cases} }
and
\an{\label{phi-2g}    \ad{
  \nabla\phi_2|_{T_1} &=  \p{  4 z - 4\\ -4 z + 4\\ 4 x - 4 y },         
      &\nabla\phi_2|_{T_3} &=   \p{  4 z \\ -4 z \\ 4 x - 4 y + 4  }, \\
       \nabla\phi_2|_{T_2} &=  \p{  4 y - 4 \\ 4 x - 4 z  \\ -4 y + 4  },      
      &\nabla\phi_2|_{T_4} &=  \p{   4 y\\ 4 x - 4 z + 4\\ -4 y  }. } }
 
On $R_2$, we need to show six cases that $p=x^3$, $p=x^2y$, $p=xy^2$,$p=y^3$, $p=y^2z$ and $p=xyz$, 
   due to a $yz$-symmetry of $R_2$.

 For $p=x^3$,  on the 4 tetrahedra $T_i$ of $R_2$,  cf. Figure \ref{3r2}, we have
\an{\label{2-p}  p-I_2 p=\begin{cases}
   x (x - 1)(2 x - 1) /2 &\t{on }\ T_1, T_2, 
   \\x (x + 1) (2 x + 1)/2 &\t{on }\ T_3,T_4, \end{cases} }
and
\an{\label{2-pd}     
  \nabla(p-I_2 p) =\begin{cases}
   \p{-3 x + 1/2 + 3 x^2 \\0\\0 } &\t{on }\ T_1, T_2, 
   \\\p{3 x^2 + 3 x + 1/2 \\0\\0 } &\t{on }\ T_3,T_4. \end{cases} }
 By \eqref{phi-2},  \eqref{phi-2g},  \eqref{2-p} and  \eqref{2-pd}, 
    we find the integral on the four tetrahedra $T_i$ of $R_2$, 
\a{  (\nabla (p- I_2 p), \nabla \phi_2)_{R_2}  
    =-\frac 1{30}-\frac 1{30}+\frac 1{30}+ \frac 1{30} =0.  }

For $p=x^2y$,  on the 4 tetrahedra $T_i$ of $R_2$,   cf. Figure \ref{3r2}, we have
\an{\label{2-p2}  p-I_2 p=\begin{cases}
   x (y - 1) (2 x - 1)/ 2 &\t{on }\ T_1, T_2, 
   \\ x y (2 x + 1)/ 2  &\t{on }\ T_3,T_4, \end{cases} }
and
\an{\label{2-p2d}     
  \nabla(p-I_2 p)=\begin{cases}
   \p{ 1/2 - 2 x -  y/2  + 2 x y \\ - x/2  + x^2 \\0 } &\t{on }\ T_1, T_2, 
   \\ \p{  y/2  + 2 x y \\  x/2  + x^2 \\0 } &\t{on }\ T_3,T_4. \end{cases} }
 By \eqref{phi-2},  \eqref{phi-2g},  \eqref{2-p2} and  \eqref{2-p2d}, 
    we find the integral on the four tetrahedra $T_i$ of $R_2$,  
\a{  (\nabla (p- I_2 p), \nabla \phi_2)_{R_2} 
   =-\frac 7{180}-\frac 1{60}+\frac 7{180}+ \frac 1{60} =0.  }
                     
For $p=x y^2$, on the 4 tetrahedra $T_i$ of $R_2$,   cf. Figure \ref{3r2}, we have
\an{\label{2-p3}  p-I_2 p=\begin{cases}
   x (y - 1) (2 y - 1)/ 2  &\t{on }\ T_1, T_2, 
   \\x y (2 y- 1)/ 2    &\t{on }\ T_3,T_4, \end{cases} }
and
\an{\label{2-p3d}     
  \nabla(p-I_2 p)=\begin{cases}
   \p{ -\frac 32 y + \frac 12 + y^2 \\ -\frac 32 x + 2 x y \\0 } &\t{on }\ T_1, T_2, 
   \\ \p{ -\frac 12 y + y^2 \\  -\frac 12 x + 2 x y\\0 } &\t{on }\ T_3,T_4. \end{cases} }
 By \eqref{phi-2},  \eqref{phi-2g},  \eqref{2-p3} and  \eqref{2-p3d}, 
    we find the integral on the four tetrahedra $T_i$ of $R_2$,  
\a{  (\nabla (p- I_2 p), \nabla \phi_2)_{R_2} 
    =-\frac 1{20}-\frac 1{180}+\frac 1{20}+ \frac 1{180} =0.  }
                     
For $p= y^3$,   on the 4 tetrahedra $T_i$ of $R_2$,   cf. Figure \ref{3r2}, we have
\an{\label{2-p4}  p-I_2 p=\frac{y \left(2 y -1\right) \left(y -1\right)}{2} 
   \ \t{on }\ T_1,T_2,T_3,T_4,  }
and
\an{\label{2-p4d}     
  \nabla(p-I_2 p)=\p{ 0 \\ \frac{1}{2}-3 y +3 y^{2} \\ 0  } 
   \ \t{on }\ T_1,T_2,T_3,T_4.  }
 By \eqref{phi-2},  \eqref{phi-2g},  \eqref{2-p4} and  \eqref{2-p4d}, 
    we find the integral on the four tetrahedra $T_i$ of $R_2$,  
\a{  (\nabla (p- I_2 p), \nabla \phi_2)_{R_2}
    =-\frac 1{60}-\frac 1{60}+\frac 1{60}+ \frac 1{60} =0.  }
              
For $p= y^2z$,  on the 4 tetrahedra $T_i$ of $R_2$,   cf. Figure \ref{3r2}, we have
\an{\label{2-p5}  p-I_2 p=\begin{cases}
    {y \left(z -1\right) \left(2 y -1\right)}/{2}  &\t{on }\ T_1, T_4, 
   \\  {z \left(2 y -1\right) \left(y -1\right)}/{2}   &\t{on }\ T_2,T_3, \end{cases} }
and
\an{\label{2-p5d}     
  \nabla(p-I_2 p)=\begin{cases}
  \p{ 0 \\ -\frac{1}{2} z +\frac{1}{2}-2 y +2 y z \\ -\frac{1}{2} y +y^{2} } &\t{on }\ T_1, T_4, 
   \\ \p{ 0 \\ -\frac{3}{2} z +2 y z  \\ -\frac{3}{2} y +\frac{1}{2}+y^{2} } &\t{on }\ T_2,T_3. \end{cases} }
 By \eqref{phi-2},  \eqref{phi-2g},  \eqref{2-p5} and  \eqref{2-p5d}, 
    we find the integral on the four tetrahedra $T_i$ of $R_2$,  
\a{  (\nabla (p- I_2 p), \nabla \phi_2)_{R_2}
     =-\frac 1{90}-\frac 1{60}+\frac 1{90}+ \frac 1{60} =0.  }

For $p= xyz$,   on the 4 tetrahedra $T_i$ of $R_2$,   cf. Figure \ref{3r2}, we have
\an{\label{2-p6}  p-I_2 p=\begin{cases} 
 {x \left(z -1\right) \left(2 y -1\right)}/{2}   &\t{on }\ T_1, \\
 {x \left(2 z -1\right) \left(y -1\right)}/{2}   &\t{on }\ T_2,\\ 
 {x z \left(2 y -1\right)}/{2}   &\t{on }\ T_3, \\ 
 {x y \left(2 z -1\right)}/{2}   &\t{on }\ T_4, \end{cases} }
and
\an{\label{2-p6d}     \ad{
  p_1'&=\p{ -\frac{1}{2} z +\frac{1}{2}-y +y z  \\ x z -x \\ -\frac{1}{2} x +x y   }, & 
 p_2' &=\p{ -\frac{1}{2} y +\frac{1}{2}-z +y z \\ -\frac{1}{2} x +x z \\ x y -x  },\\ 
  p_3' &=\p{ -\frac{1}{2} z +y z \\ x z  \\-\frac{1}{2} x +x y   }, & 
 p_4' &=\p{ -\frac{1}{2} y +y z \\ -\frac{1}{2} x +x z  \\ x y     }, } } 
where $p_i'=\nabla(p-I_2 p)|_{T_i}$, $i=1,2,3,4$.
 By \eqref{phi-2},  \eqref{phi-2g},  \eqref{2-p6} and  \eqref{2-p6d}, 
    we find the integral on the four tetrahedra $T_i$ of $R_2$,  
\a{  (\nabla (p- I_2 p), \nabla \phi_2)_{R_2}
   =-\frac 1{30}-\frac 1{36}+\frac 1{30}+ \frac 1{36} =0.  }
  
  The proof is complete.
\end{proof}

\begin{lemma} 
Let $p\in P_3(R_3)$,   where $R_3$ is the support patch of tetrahedra of
  a mid-vertical-edge node basis function $\phi_3$, cf. Figure \ref{3r3}.
    Then the $P_2$ nodal interpolation of $p$, $I_2 p$, is the local $H^1$-projection of $p$,
     i.e., 
\an{\label{d-3-3} (\nabla(p-I_2 p), \nabla \phi_3)_{R_3} = 0.  }  
 \end{lemma} 

\begin{proof}  We shift and isotropic scale the patch  $R_3$ to
  the one  shown in Figure \ref{3r3}.

 \begin{figure}[H] 
  \setlength{\unitlength}{1pt}

 \begin{center}\begin{picture}(200.,120 )(-20., 0.)

 \put(0,0){\begin{picture}(200.,120)(  0.,  0.)   
    \multiput(0,20)(4,0){20}{\circle*{1}} \multiput(80,20)(4,-2){10}{\circle*{1}} 
     \multiput(80,20)(0,4){20}{\circle*{1}}  \multiput(80,20)(-3,3){26}{\circle*{1}} 
      \multiput(80,20)(2,3){20}{\circle*{1}} \multiput(80,20)(-4,-2){10}{\circle*{1}}
      \multiput(80,20)(-2,3){20}{\circle*{1}} 
       
       \put(80,60){\circle*{4}} \put(75,108){$\b x_5$} \put(75,10){$\b x_1$}
      
     \put(0,100){\line(1,0){80}}\put(0,100){\line(0,-1){80}}\put(0,100){\line(2,-1){40}}
     \put(40, 0){\line(-2, 1){40}}\put(40, 0){\line(0, 1){80}}\put(40, 0){\line(1,0){80}}
     \put(120,80){\line(-2, 1){40}}\put(120, 80){\line(0,-1){80}}\put(120, 80){\line(-1,0){80}}
     \put(40,80){\line(2, 1){40}}\put(40, 80){\line(-2,-3){40}}\put(40, 80){\line( 1,-1){80}}

 \put(-40,20){\begin{picture}(120.,120)(  0.,  0.)   
    \multiput(0,20)(4,0){20}{\circle*{1}} \multiput(80,20)(4,-2){10}{\circle*{1}} 
     \multiput(80,20)(0,4){20}{\circle*{1}}  \multiput(80,20)(-3,3){26}{\circle*{1}} 
      \multiput(80,20)(2,3){20}{\circle*{1}} \multiput(80,20)(-4,-2){10}{\circle*{1}}
      \multiput(80,20)(-2,3){20}{\circle*{1}}

     \put(0,100){\line(1,0){80}}\put(0,100){\line(0,-1){80}}\put(0,100){\line(2,-1){40}}
     \put(40, 0){\line(-2, 1){40}}\put(40, 0){\line(0, 1){80}} 
     \put(120,80){\line(-2, 1){40}} \put(120, 80){\line(-1,0){80}}
     \put(40,80){\line(2, 1){40}}\put(40, 80){\line(-2,-3){40}} 

       \end{picture} }
 \put(80,0){\begin{picture}(120.,120)(  0.,  0.)   
    \multiput(0,20)(4,0){20}{\circle*{1}} \multiput(80,20)(4,-2){10}{\circle*{1}} 
     \multiput(80,20)(0,4){20}{\circle*{1}}  \multiput(80,20)(-3,3){26}{\circle*{1}} 
      \multiput(80,20)(2,3){20}{\circle*{1}} \multiput(80,20)(-4,-2){10}{\circle*{1}}
      \multiput(80,20)(-2,3){20}{\circle*{1}}

     \put(0,100){\line(1,0){80}} \put(0,100){\line(2,-1){40}}
      \put(40, 0){\line(0, 1){80}} \put(40, 0){\line(1,0){80}}
     \put(120,80){\line(-2, 1){40}}\put(120, 80){\line(0,-1){80}}\put(120, 80){\line(-1,0){80}}
     \put(40,80){\line(2, 1){40}} \put(40, 80){\line( 1,-1){80}} 
     
       \end{picture} }

 \put(40,20){\begin{picture}(120.,120)(  0.,  0.)   
    \multiput(0,20)(4,0){20}{\circle*{1}} \multiput(80,20)(4,-2){10}{\circle*{1}} 
     \multiput(80,20)(0,4){20}{\circle*{1}}  \multiput(80,20)(-3,3){26}{\circle*{1}} 
      \multiput(80,20)(2,3){20}{\circle*{1}} \multiput(80,20)(-4,-2){10}{\circle*{1}}
      \multiput(80,20)(-2,3){20}{\circle*{1}}

     \put(0,100){\line(1,0){80}} \put(0,100){\line(2,-1){40}} 
     \put(120,80){\line(-2, 1){40}} \put(120, 80){\line(-1,0){80}}
     \put(40,80){\line(2, 1){40}}  
     
       \end{picture} }

       \end{picture} }

 \end{picture}\end{center}

\caption{The support patch $R_3$ of tetrahedra of the $P_2$ Lagrange nodal basis $\phi_3$ at
   a mid-vertical edge node, where $\b x_1(0,0,0)$, $\b x_2(-1,0,0)$, $\b x_3(-1,-1,0)$,
      $\b x_4(0,-1,0)$, $\b x_5(0,0,1)$, $\b x_6(1,0,1)$,  $\b x_7(1,1,1)$,  $\b x_8(0,1,1)$, 
        $T_1=\b x_1\b x_5\b x_6\b x_7$, $T_2=\b x_1\b x_5\b x_7\b x_8$,
         $T_3=\b x_1\b x_ 2\b x_5\b x_ 8$, $T_4=\b x_ 1\b x_2\b x_3\b x_5$,
         $T_5=\b x_ 1\b x_3\b x_4\b x_5$ and $T_6=\b x_ 1\b x_4\b x_5\b x_6$. }
\label{3r3}
\end{figure}
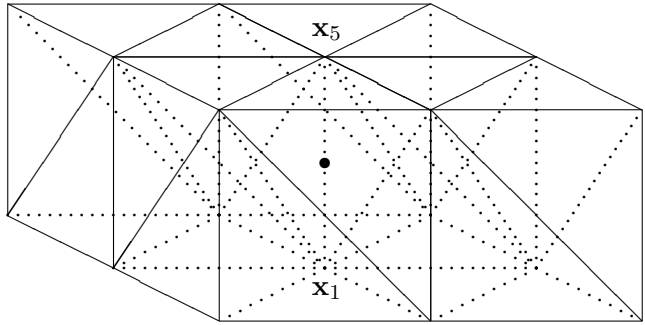

 On the six tetrahedra of $R_3$,  we find
\an{\label{phi-3} \phi_3=\begin{cases}
   4 \left(z -1\right) \left(x -z \right)  &\t{on }\ T_1, \\
4 \left(z -1\right) \left(y -z \right) &\t{on }\ T_2,\\
-4 \left(y -z \right) \left(x -z +1\right) &\t{on }\ T_3,\\
4 \left(x -z +1\right) z &\t{on }\ T_4, \\
4 z \left(y -z +1\right) &\t{on }\ T_5,\\
-4 \left(y -z +1\right) \left(x -z \right) &\t{on }\ T_6, \end{cases} }
and
\an{\label{phi-3g}    \ad{
  \nabla\phi_3|_{T_1} &=  \p{ 4 z -4\\ 0\\ 4 x -8 z +4 },         
      &\nabla\phi_3|_{T_2} &=  \p{     0 \\ 4 z -4 \\4 y -8 z +4 }, \\
       \nabla\phi_3|_{T_3} &=  \p{  -4 y +4 z \\ -4 x +4 z -4\\ 4 x +4 y -8 z +4 }, 
       & \nabla\phi_3|_{T_4} &= \p{  4 z \\ 0\\ 4 x -8 z +4 },  \\    
       \nabla\phi_3|_{T_5} &=   \p{ 0\\ 4 z \\ 4 y -8 z +4 }, 
    &\nabla\phi_3|_{T_6} &= \p{ -4 y +4 z -4 \\ -4 x +4 z \\ 4 x +4 y -8 z +4 }. } }

On $R_3$, we need to show three cases that $p=x^3$, $p=x^2y$, $p=x^2z$,
     $p=z^3$, $p=x z^2$ and $p=xyz$, 
   due to an $xy$-symmetry of $R_3$.

 For $p=x^3$,  on the six tetrahedra $T_i$ of $R_3$,  cf. Figure \ref{3r3},  we have
\an{\label{3-p1}  p-I_2 p=\begin{cases}
   {x \left(2 x -1\right) \left(x -1\right)}/{2}&\t{on }\ T_1, T_2, T_6,
   \\ {x \left(1+x \right) \left(2 x +1\right)}/{2} &\t{on }\ T_3,T_4,T_5, \end{cases} }
and
\an{\label{3-p1d}     
  \nabla(p-I_2 p) =\begin{cases}
   \p{-3 x +\frac{1}{2}+3 x^{2}\\0 \\ 0 } &\t{on }\ T_1, T_2, T_6,
   \\\p{ 3 x^{2}+3 x +\frac{1}{2} \\ 0 \\ 0   } &\t{on }\ T_3,T_4,T_5. \end{cases} }
 By \eqref{phi-3},  \eqref{phi-3g},  \eqref{3-p1} and  \eqref{3-p1d}, 
    we find the integral on the six tetrahedra $T_i$ of $R_3$, 
\a{  (\nabla (p- I_2 p), \nabla \phi_3)_{R_3} 
   = \frac 1{30}+0+\frac 1{30}-\frac 1{30}+0-\frac 1{30} =0.  }

For $p=x^2y$,   on the six tetrahedra $T_i$ of $R_3$,  cf. Figure \ref{3r3},  we have
\an{\label{3-p2}  p-I_2 p=\begin{cases}
  {y \left(2 x -1\right) \left(x -1\right)}/{2}&\t{on }\ T_1,\\
  {x \left(y -1\right) \left(2 x -1\right)} /{2}&\t{on }\  T_2,\\
  {x y \left(2 x +1\right)} /{2}&\t{on }\  T_3,\\
  {y \left(1+x \right) \left(2 x +1\right)} /{2}&\t{on }\ T_4,\\
  {x \left(y +1\right) \left(2 x +1\right)} /{2}&\t{on }\ T_5,\\
  {x y \left(2 x -1\right)} /{2}&\t{on }\ T_6,  \end{cases} }
and
\an{\label{3-p2d} \ad{  g_{1} &= \p{-\frac{3}{2} y +2 x y \\-\frac{3}{2} x +\frac{1}{2}+x^{2}\\ 0 }, 
  &   g_{2} &=\p{\frac{1}{2}-2 x -\frac{1}{2} y +2 x y \\ -\frac{1}{2} x +x^{2} \\ 0  },  \\
 g_{3} &= \p{ \frac{1}{2} y +2 x y \\ \frac{1}{2} x +x^{2} \\ 0  }, 
  &   g_{4} &= \p{\frac{3}{2} y +2 x y \\ \frac{3}{2} x +\frac{1}{2}+x^{2} \\ 0   },  \\
   g_{5} &= \p{ 2 x +\frac{1}{2} y +\frac{1}{2}+2 x y \\ \frac{1}{2} x +x^{2} \\ 0   }, 
  &   g_{6} &= \p{ -\frac{1}{2} y +2 x y  \\-\frac{1}{2} x +x^{2}  \\0  },  } }
  where $g_{i}=\nabla(p-I_2p)|_{T_i}$.
 By \eqref{phi-3},  \eqref{phi-3g},  \eqref{3-p2} and  \eqref{3-p2d}, 
    we find the integral on the six tetrahedra $T_i$ of $R_3$, 
\a{  (\nabla (p- I_2 p), \nabla \phi_3)_{R_3}  
  = \frac 1{30}+ \frac 1{180}+ \frac 1{60}-\frac 1{30}-\frac 1{180}-\frac 1{60} =0.  }

For $p=x^2z$,   on the six tetrahedra $T_i$ of $R_3$,  cf. Figure \ref{3r3},  we have
\an{\label{3-p3}  p-I_2 p=\begin{cases}
  {x \left(z -1\right) \left(2 x -1\right)}/{2}&\t{on }\ T_1,  T_2, T_6, \\
  {x z \left(2 x +1\right)}/{2}&\t{on }\  T_3,  T_4,  T_5,  \end{cases} }
and
\an{\label{3-p3d}\nabla(p-I_2p)=\begin{cases}
  \p{-2 x -\frac{1}{2} z +\frac{1}{2}+2 x z \\ 0 \\ -\frac{1}{2} x +x^{2} }&\t{on }\ T_1,  T_2, T_6, \\
  \p{ \frac{1}{2} z +2 x z  \\ 0 \\ \frac{1}{2} x +x^{2}  }&\t{on }\  T_3,  T_4,  T_5.  \end{cases} } 
 By \eqref{phi-3},  \eqref{phi-3g},  \eqref{3-p3} and  \eqref{3-p3d}, 
    we find the integral on the six tetrahedra $T_i$ of $R_3$, 
\a{  (\nabla (p- I_2 p), \nabla \phi_3)_{R_3} 
   = \frac 1{45}+ 0+ \frac 1{30}-\frac 1{45}-0-\frac 1{30} =0.  }

For $p=x^2z$,   on the six tetrahedra $T_i$ of $R_3$,  cf. Figure \ref{3r3},  we have
\an{\label{3-p4}  p-I_2 p= {z \left(2 z -1\right) \left(z -1\right)}/{2} \quad \
          \t{on }\ T_i, i=1,\dots,6, }
and
\an{\label{3-p4d}\nabla(p-I_2p)=\p{0 \\ 0 \\ -3 z +\frac{1}{2}+3 z^{2} }\quad \
          \t{on }\ T_i, i=1,\dots,6. }
 By \eqref{phi-3},  \eqref{phi-3g},  \eqref{3-p4} and  \eqref{3-p4d}, 
    we find the integral on the six tetrahedra $T_i$ of $R_3$, 
\a{  (\nabla (p- I_2 p), \nabla \phi_3)_{R_3}  = 0+0+0+0+0+ 0  =0.  }

For $p=x z^2$,    on the six tetrahedra $T_i$ of $R_3$,  cf. Figure \ref{3r3},  we have
\an{\label{3-p5}  p-I_2 p=\begin{cases}
  {x \left(2 z -1\right) \left(z -1\right)}/{2}&\t{on }\ T_1,  T_2, T_6, \\
  {x z \left(2 z -1\right)} /{2}&\t{on }\  T_3,  T_4,  T_5,  \end{cases} }
and
\an{\label{3-p5d}\nabla(p-I_2p)=\begin{cases}
  \p{ -\frac{3}{2} z +\frac{1}{2}+z^{2} \\ 0 \\ -\frac{3}{2} x +2 x z    } &\t{on }\ T_1,  T_2, T_6, \\
 \p{  -\frac{1}{2} z +z^{2} \\ 0 \\ -\frac{1}{2} x +2 x z    } &\t{on }\  T_3,  T_4,  T_5.  \end{cases} } 
 By \eqref{phi-3},  \eqref{phi-3g},  \eqref{3-p5} and  \eqref{3-p5d}, 
    we find the integral on the six tetrahedra $T_i$ of $R_3$, 
\a{  (\nabla (p- I_2 p), \nabla \phi_3)_{R_3}  = 0+0+0+0+0+ 0  =0.  }

For $p=xyz$,  on the six tetrahedra $T_i$ of $R_3$,  cf. Figure \ref{3r3},  we have
\an{\label{3-p6}  p-I_2 p=\begin{cases}
  {y \left(z -1\right) \left(2 x -1\right)} /{2}&\t{on }\ T_1,\\
  {x \left(z -1\right) \left(2 y -1\right)} /{2}&\t{on }\  T_2,\\
  {x y \left(2 z -1\right)} /{2}&\t{on }\  T_3,\\
  {y z \left(2 x +1\right)} /{2}&\t{on }\ T_4,\\
  {x z \left(1+2 y \right)} /{2}&\t{on }\ T_5,\\
  {x y \left(2 z -1\right)} /{2}&\t{on }\ T_6,  \end{cases} }
and
\an{\label{3-p6d} \ad{  g_{1} &= 
   \p{ y z -y  \\ -\frac{1}{2} z +\frac{1}{2}-x +x z  \\ -\frac{1}{2} y +x y   },
  &   g_{2} &=\p{   -\frac{1}{2} z +\frac{1}{2}-y +y z  \\ x z -x  \\ -\frac{1}{2} x +x y   },   \\
      g_{3} &= \p{ -\frac{1}{2} y +y z  \\ -\frac{1}{2} x +x z  \\ x y    }, 
  &   g_{4} &= \p{  y z  \\ \frac{1}{2} z +x z  \\ \frac{1}{2} y +x y   },   \\
   g_{5} &= \p{  \frac{1}{2} z +y z  \\ x z  \\ \frac{1}{2} x +x y   },  
  &   g_{6} &= \p{ -\frac{1}{2} y +y z  \\ -\frac{1}{2} x +x z  \\ x y      },   } }
  where $g_{i}=\nabla(p-I_2p)|_{T_i}$.
 By \eqref{phi-3},  \eqref{phi-3g},  \eqref{3-p6} and  \eqref{3-p6d}, 
    we find the integral on the six tetrahedra $T_i$ of $R_3$, 
\a{  (\nabla (p- I_2 p), \nabla \phi_3)_{R_3}  
  = \frac 1{45}+ \frac 1{90}+ \frac 1{180}- \frac 1{45}- \frac 1{90}-\frac 1{180}  =0.  }

  The proof is complete.
\end{proof}

\begin{lemma} 
Let $p\in P_3(R_4)$,   where $R_4$ is the support patch of tetrahedra of
  a vertex node basis function $\phi_4$, cf. Figure \ref{3r4}.
    Then the $P_2$ nodal interpolation of $p$, $I_2 p$, is the local $H^1$-projection of $p$,
     i.e., 
\an{\label{d-3-4} (\nabla(p-I_2 p), \nabla \phi_4)_{R_4} = 0.  }  
 \end{lemma} 

\begin{proof}  We shift and isotropic scale the patch  $R_4$ to
  the one  shown in Figure \ref{3r4}.

 \begin{figure}[H] 
  \setlength{\unitlength}{0.9pt}

 \begin{center}\begin{picture}(200.,200 )(-10., 0.)

 \put(0, 0){\begin{picture}(200.,200)(  0.,  0.)   
    \multiput(0,20)(4,0){20}{\circle*{1}} \multiput(80,20)(4,-2){10}{\circle*{1}} 
     \multiput(80,20)(0,4){20}{\circle*{1}}  \multiput(80,20)(-3,3){26}{\circle*{1}} 
      \multiput(80,20)(2,3){20}{\circle*{1}} \multiput(80,20)(-4,-2){10}{\circle*{1}}
      \multiput(80,20)(-2,3){20}{\circle*{1}} 
       
       \put(80,100){\circle*{4}}  
      
    \put(0,100){\line(0,-1){80}} 
     \put(40, 0){\line(-2, 1){40}}\put(40, 0){\line(0, 1){80}}\put(40, 0){\line(1,0){80}}
      \put(120, 80){\line(0,-1){80}} 
      \put(40, 80){\line(-2,-3){40}}\put(40, 80){\line( 1,-1){80}}

 \put(-40,20){\begin{picture}(120.,120)(  0.,  0.)   
    \multiput(0,20)(4,0){20}{\circle*{1}} \multiput(80,20)(4,-2){10}{\circle*{1}} 
     \multiput(80,20)(0,4){20}{\circle*{1}}  \multiput(80,20)(-3,3){26}{\circle*{1}} 
      \multiput(80,20)(2,3){20}{\circle*{1}} \multiput(80,20)(-4,-2){10}{\circle*{1}}
      \multiput(80,20)(-2,3){20}{\circle*{1}} 
       
       \put(0,100){\line(0,-1){80}} 
     \put(40, 0){\line(-2, 1){40}}\put(40, 0){\line(0, 1){80}}   \put(40, 80){\line(-2,-3){40}} 

       \end{picture} }
 \put(80,0){\begin{picture}(120.,120)(  0.,  0.)   
    \multiput(0,20)(4,0){20}{\circle*{1}} \multiput(80,20)(4,-2){10}{\circle*{1}} 
     \multiput(80,20)(0,4){20}{\circle*{1}}  \multiput(80,20)(-3,3){26}{\circle*{1}} 
      \multiput(80,20)(2,3){20}{\circle*{1}} \multiput(80,20)(-4,-2){10}{\circle*{1}}
      \multiput(80,20)(-2,3){20}{\circle*{1}}

      \put(40, 0){\line(0, 1){80}} \put(40, 0){\line(1,0){80}}
      \put(120, 80){\line(0,-1){80}}   \put(40, 80){\line( 1,-1){80}} 
     
       \end{picture} }

 \put(40,20){\begin{picture}(120.,120)(  0.,  0.)   
    \multiput(0,20)(4,0){20}{\circle*{1}} \multiput(80,20)(4,-2){10}{\circle*{1}} 
     \multiput(80,20)(0,4){20}{\circle*{1}}  \multiput(80,20)(-3,3){26}{\circle*{1}} 
      \multiput(80,20)(2,3){20}{\circle*{1}} \multiput(80,20)(-4,-2){10}{\circle*{1}}
      \multiput(80,20)(-2,3){20}{\circle*{1}}

       \end{picture} }

       \end{picture} }

 \put(0,80){\begin{picture}(200.,200)(  0.,  0.)   
    \multiput(0,20)(4,0){20}{\circle*{1}} \multiput(80,20)(4,-2){10}{\circle*{1}} 
     \multiput(80,20)(0,4){20}{\circle*{1}}  \multiput(80,20)(-3,3){26}{\circle*{1}} 
      \multiput(80,20)(2,3){20}{\circle*{1}} \multiput(80,20)(-4,-2){10}{\circle*{1}}
      \multiput(80,20)(-2,3){20}{\circle*{1}} 
       
       \put(70,108){$\b x_{14}$}   \put(72,-71){$\b x_1$}   \put(-13,28){$\b x_5$} 
      
     \put(0,100){\line(1,0){80}}\put(0,100){\line(0,-1){80}}\put(0,100){\line(2,-1){40}}
     \put(40, 0){\line(-2, 1){40}}\put(40, 0){\line(0, 1){80}}\put(40, 0){\line(1,0){80}}
     \put(120,80){\line(-2, 1){40}}\put(120, 80){\line(0,-1){80}}\put(120, 80){\line(-1,0){80}}
     \put(40,80){\line(2, 1){40}}\put(40, 80){\line(-2,-3){40}}\put(40, 80){\line( 1,-1){80}}

 \put(-40,20){\begin{picture}(120.,120)(  0.,  0.)   
    \multiput(0,20)(4,0){20}{\circle*{1}} \multiput(80,20)(4,-2){10}{\circle*{1}} 
     \multiput(80,20)(0,4){20}{\circle*{1}}  \multiput(80,20)(-3,3){26}{\circle*{1}} 
      \multiput(80,20)(2,3){20}{\circle*{1}} \multiput(80,20)(-4,-2){10}{\circle*{1}}
      \multiput(80,20)(-2,3){20}{\circle*{1}}

     \put(0,100){\line(1,0){80}}\put(0,100){\line(0,-1){80}}\put(0,100){\line(2,-1){40}}
     \put(40, 0){\line(-2, 1){40}}\put(40, 0){\line(0, 1){80}} 
     \put(120,80){\line(-2, 1){40}} \put(120, 80){\line(-1,0){80}}
     \put(40,80){\line(2, 1){40}}\put(40, 80){\line(-2,-3){40}} 

       \end{picture} }
 \put(80,0){\begin{picture}(120.,120)(  0.,  0.)   
    \multiput(0,20)(4,0){20}{\circle*{1}} \multiput(80,20)(4,-2){10}{\circle*{1}} 
     \multiput(80,20)(0,4){20}{\circle*{1}}  \multiput(80,20)(-3,3){26}{\circle*{1}} 
      \multiput(80,20)(2,3){20}{\circle*{1}} \multiput(80,20)(-4,-2){10}{\circle*{1}}
      \multiput(80,20)(-2,3){20}{\circle*{1}}

     \put(0,100){\line(1,0){80}} \put(0,100){\line(2,-1){40}}
      \put(40, 0){\line(0, 1){80}} \put(40, 0){\line(1,0){80}}
     \put(120,80){\line(-2, 1){40}}\put(120, 80){\line(0,-1){80}}\put(120, 80){\line(-1,0){80}}
     \put(40,80){\line(2, 1){40}} \put(40, 80){\line( 1,-1){80}} 
     
       \end{picture} }

 \put(40,20){\begin{picture}(120.,120)(  0.,  0.)   
    \multiput(0,20)(4,0){20}{\circle*{1}} \multiput(80,20)(4,-2){10}{\circle*{1}} 
     \multiput(80,20)(0,4){20}{\circle*{1}}  \multiput(80,20)(-3,3){26}{\circle*{1}} 
      \multiput(80,20)(2,3){20}{\circle*{1}} \multiput(80,20)(-4,-2){10}{\circle*{1}}
      \multiput(80,20)(-2,3){20}{\circle*{1}}

     \put(0,100){\line(1,0){80}} \put(0,100){\line(2,-1){40}} 
     \put(120,80){\line(-2, 1){40}} \put(120, 80){\line(-1,0){80}}
     \put(40,80){\line(2, 1){40}}  
     
       \end{picture} }

       \end{picture} }
 
 \end{picture}\end{center}

\caption{The support patch $R_4$ of tetrahedra of the $P_2$ Lagrange nodal basis $\phi_4$ at
   a vertex node $\b x_{15}$, where  
$\b x_{1}(0,0,-1)$, 
$\b x_{2}(-1,0,-1)$, 
$\b x_{3}(-1,-1,-1)$, 
$\b x_{4}(0,-1,-1)$, 
$\b x_{5}(1,0,0)$, 
$\b x_{6}(1,1,0)$, 
$\b x_{7}(0,1,0)$, 
$\b x_{8}(-1,0,0)$, 
$\b x_{9}(-1,-1,0)$, 
$\b x_{10}(0,-1,0)$, 
$\b x_{11}(1,0,1)$, 
$\b x_{12}(1,1,1)$, 
$\b x_{13}(0,1,1)$, 
$\b x_{14}(0,0,1)$, 
$\b x_{15}(0,0,0)$, 
$T_{1}=\b x_{1}\b x_{5}\b x_{6}\b x_{15}$, 
$T_{2}=\b x_{1}\b x_{6}\b x_{7}\b x_{15}$, 
$T_{3}=\b x_{1}\b x_{2}\b x_{7}\b x_{15}$, 
$T_{4}=\b x_{7}\b x_{2}\b x_{8}\b x_{15}$, 
$T_{5}=\b x_{1}\b x_{2}\b x_{3}\b x_{15}$, 
$T_{6}=\b x_{2}\b x_{3}\b x_{8}\b x_{15}$, 
$T_{7}=\b x_{3}\b x_{8}\b x_{9}\b x_{15}$, 
$T_{8}=\b x_{3}\b x_{9}\b x_{10}\b x_{15}$, 
$T_{9}=\b x_{3}\b x_{10}\b x_{4}\b x_{15}$, 
$T_{10}=\b x_{3}\b x_{1}\b x_{4}\b x_{15}$, 
$T_{11}=\b x_{1}\b x_{4}\b x_{5}\b x_{15}$, 
$T_{12}=\b x_{10}\b x_{4}\b x_{5}\b x_{15}$, 
$T_{13}=\b x_{5}\b x_{6}\b x_{12}\b x_{15}$, 
$T_{14}=\b x_{6}\b x_{7}\b x_{12}\b x_{15}$, 
$T_{15}=\b x_{7}\b x_{13}\b x_{12}\b x_{15}$, 
$T_{16}=\b x_{12}\b x_{13}\b x_{14}\b x_{15}$, 
$T_{17}=\b x_{11}\b x_{12}\b x_{14}\b x_{15}$, 
$T_{18}=\b x_{5}\b x_{11}\b x_{12}\b x_{15}$, 
$T_{19}=\b x_{7}\b x_{8}\b x_{13}\b x_{15}$, 
$T_{20}=\b x_{8}\b x_{13}\b x_{14}\b x_{15}$, 
$T_{21}=\b x_{8}\b x_{9}\b x_{14}\b x_{15}$, 
$T_{22}=\b x_{9}\b x_{10}\b x_{14}\b x_{15}$, 
$T_{23}=\b x_{10}\b x_{11}\b x_{14}\b x_{15}$, and 
$T_{24}=\b x_{10}\b x_{5}\b x_{11}\b x_{15}$ . }
\label{3r4}
\end{figure}
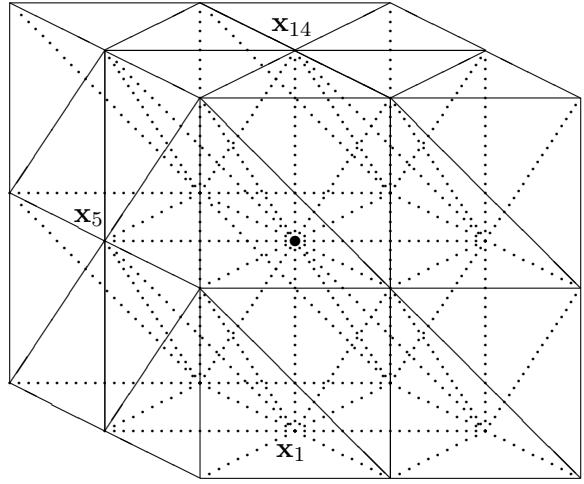

The nodal basis $\phi_4$ is supported on the 24 tetrahedra that $\phi_4=$
\a{ \begin{cases}  
\left(1+z \right) \left(1+2 z \right),&T_{5},T_{10}, \\  
\left(1+x \right) \left(2 x +1\right),&T_{6},T_{7}, \\ 
\left(y +1\right) \left(1+2 y \right),&T_{8},T_{9}, \\ 
\left(2 x -1\right) \left(x -1\right),&T_{13},T_{18}, \\ 
\left(2 y -1\right) \left(y -1\right),&T_{14},T_{15}, \\ 
\left(2 z -1\right) \left(z -1\right),&T_{16},T_{17},  
    \end{cases}  \begin{cases}  
\left(2 x -2 y -1\right) \left(x -y -1\right),&T_{12},T_{24}, \\ 
\left(2 x -2 z -1\right) \left(x -z -1\right),&T_{1},T_{11}, \\ 
\left(2 y -2 z -1\right) \left(y -z -1\right),&T_{2},T_{3}, \\ 
\left(x -y +1\right) \left(2 x -2 y +1\right),&T_{4},T_{19}, \\ 
\left(x -z +1\right) \left(2 x -2 z +1\right),&T_{20},T_{21} , \\ 
\left(y -z +1\right) \left(2 y -2 z +1\right),&T_{22},T_{ 23},  
    \end{cases}}  \def\lra#1{\langle #1 \rangle}
and $\nabla\phi_4=$
\a{ \begin{cases}  
  \lra{ 0,0,4 z +3},&T_{5},T_{10}, \\  
  \lra{ 4 x +3,0,0},&T_{6},T_{7}, \\ 
  \lra{ 0,4 y +3,0},&T_{8},T_{9}, \\ 
  \lra{ 4 x -3,0,0},&T_{13},T_{18}, \\ 
  \lra{ 0,4 y -3,0},&T_{14},T_{15}, \\ 
  \lra{ 0,0,4 z -3},&T_{16},T_{17},  
    \end{cases}  \begin{cases}  
  \lra{ 4 x -4 y -3,-4 x +4 y +3,0},&T_{12},T_{24}, \\ 
\lra{ 4 x -4 z -3,0,-4 x +4 z +3},&T_{1},T_{11}, \\ 
  \lra{ 0,4 y -4 z -3,-4 y +4 z +3},&T_{2},T_{3}, \\ 
 \lra{ 4 x -4 y +3,-4 x +4 y -3,0},&T_{4},T_{19}, \\ 
  \lra{ 4 x -4 z +3,0,-4 x +4 z -3},&T_{20},T_{21} , \\ 
  \lra{ 0,4 y -4 z +3,-4 y +4 z -3},&T_{22},T_{ 23}.  
    \end{cases}}  

For $p=x^3$,  on the 24 tetrahedra of $R_4$,  we have
\a{ p-I_2p=\begin{cases} 
    \frac{x \left(2 x -1\right) \left(x -1\right)}{2}, & \t{on} \ T_i,\ i=1,2,11,\dots,18,23,24, \\  
    \frac{x \left(1+x \right) \left(2 x +1\right)}{2}, & \t{on} \ T_i,\ i=3,\dots,10,19,\dots,22,
    \end{cases} } 
and $\nabla(p-I_2p)=$
\an{\label{4-p1d} \begin{cases} 
    \lra{-3 x +\frac{1}{2}+3 x^{2},0,0}, & T_i,\ i=1,2,11,\dots,18,23,24, \\  
    \lra{3 x^{2}+3 x +\frac{1}{2},0,0}, & T_i,\ i=3,\dots,10,19,\dots,22.
    \end{cases} } 
 By $\nabla \phi_4$ above and  \eqref{4-p1d}, 
    we find the integral on the 24 tetrahedra $T_i$ of $R_4$, 
\a{ (\nabla(p-I_2p),\nabla \phi_4)_{T_i}&=\begin{cases}  
    -\frac 1{20}, & i=6,7,\\
    \ \; \frac 1{20}, & i=13,18, \\
    -\frac 1{60}, & i=11,12, \\ 
    \ \;  \frac 1{60}, & i=19,20, \\
     0, & \t{rest} \ i.
    \end{cases} \\
    (\nabla(p-I_2p),\nabla \phi_4)_{R_4} &= 0.  }

For $p=x^2y$,   on the 24 tetrahedra of $R_4$,  we have
\a{ p-I_2p=\begin{cases} 
    \frac{y \left(2 x -1\right) \left(x -1\right)}{2},& T_i,\ i=1,13,17,18,  \\  
    \frac{x \left(y -1\right) \left(2 x -1\right)}{2},& T_i,\ i=2,14,15,16, \\ 
    \frac{x y \left(2 x +1\right)}{2},& T_i,\ i=3,4,19,20, \\ 
    \frac{y \left(1+x \right) \left(2 x +1\right)}{2},& T_i,\ i=5,6,7,21, \\ 
    \frac{x \left(y +1\right) \left(2 x +1\right)}{2},& T_i,\ i=8,9,10,22, \\ 
    \frac{x y \left(2 x -1\right)}{2},& T_i,\ i=11,12,23,24. 
    \end{cases} }  and $\nabla(p-I_2p)=$
\an{\label{4-p2d} \begin{cases} 
    \lra{-\frac{3}{2} y +2 x y ,-\frac{3}{2} x +\frac{1}{2}+x^{2},0 },& T_i,\ i=1,13,17,18,  \\  
    \lra{-2 x -\frac{1}{2} y +\frac{1}{2}+2 x y ,-\frac{1}{2} x +x^{2},0},& T_i,\ i=2,14,15,16, \\ 
    \lra{\frac{1}{2} y +2 x y ,\frac{1}{2} x +x^{2},0},& T_i,\ i=3,4,19,20, \\ 
    \lra{\frac{3}{2} y +2 x y ,\frac{3}{2} x +\frac{1}{2}+x^{2},0},& T_i,\ i=5,6,7,21, \\ 
    \lra{\frac{1}{2}+2 x +\frac{1}{2} y +2 x y ,\frac{1}{2} x +x^{2},0},& T_i,\ i=8,9,10,22, \\ 
    \lra{-\frac{1}{2} y +2 x y ,-\frac{1}{2} x +x^{2},0},& T_i,\ i=11,12,23,24. 
    \end{cases} }  
By $\nabla \phi_4$ above and  \eqref{4-p2d}, 
    we find the integral on the 24 tetrahedra $T_i$ of $R_4$,  
$(\nabla(p-I_2p),\nabla \phi_4)_{T_i}=$
\a{\begin{cases}  
   \ \;  \frac 1{120}, & i=1,14,\\
    -\frac 1{120}, & i=8,21,\\
 \ \;  \frac 1{45}, & i=4,\\
   -\frac 1{45}, & i=24,\\
  \ \;   \frac 7{360}, & i=18,\\
    -\frac 7{360}, & i=6,\\
 \ \;  \frac 7{180}, & i=13,\\
    -\frac 7{180}, & i=7, 
    \end{cases} && \begin{cases}  
 \ \;  \frac 1{720}, & i=2,3,15,\\
    -\frac 1{720}, & i=9,22,23,\\
 \ \;  \frac 1{360}, & i=20,\\
    -\frac 1{360}, & i=11,\\
 \ \;  \frac 1{144}, & i=19,\\
    -\frac 1{144}, & i=12,\\ \\
 \ \;  0, & \t{rest} \ i.
    \end{cases} }
Thus $ (\nabla(p-I_2p),\nabla \phi_4)_{R_4}  = 0$.

For $p=x y z$,  on the 24 tetrahedra of $R_4$,  we have $p-I_2p=$
\a{ \begin{cases} 
    \frac{x y \left(1+2 z \right)}{2},& T_{3},T_{11},  \\  
    \frac{x y \left(2 z -1\right)}{2},& T_{20},T_{23},  \\  
    \frac{x z \left(2 y -1\right)}{2},& T_{2},T_{19},  \\   
    \frac{x z \left(1+2 y \right)}{2},& T_{12},T_{22},  \\   
    \frac{y z \left(2 x -1\right)}{2},& T_{1},T_{24},  \\ 
    \frac{y z \left(2 x +1\right)}{2},& T_{4},T_{21},  \\   
    \frac{x \left(1+2 z \right) \left(y +1\right)}{2},& T_{9},  \\  
    \frac{x \left(1+z \right) \left(1+2 y \right)}{2},& T_{10},  \\
    \frac{x \left(2 z -1\right) \left(y -1\right)}{2},& T_{15}, 
     \end{cases}&&  \begin{cases} 
    \frac{x \left(z -1\right) \left(2 y -1\right)}{2},& T_{16},  \\  
    \frac{y \left(1+z \right) \left(2 x +1\right)}{2},& T_{5},  \\  
    \frac{y \left(1+2 z \right) \left(1+x \right)}{2},& T_{6},  \\   
    \frac{y \left(z -1\right) \left(2 x -1\right)}{2},& T_{17},  \\  
    \frac{y \left(2 z -1\right) \left(x -1\right)}{2},& T_{18},  \\ 
    \frac{z \left(1+2 y \right) \left(1+x \right)}{2},& T_{7},  \\  
    \frac{z \left(y +1\right) \left(2 x +1\right)}{2},& T_{8},  \\    
    \frac{z \left(2 y -1\right) \left(x -1\right)}{2},& T_{13},  \\  
    \frac{z \left(y -1\right) \left(2 x -1\right)}{2},& T_{14},
    \end{cases} }  
and $\nabla(p-I_2p)=$ 
\an{\label{4-p3d} \begin{cases} 
    \lra{\frac{1}{2} y +y z ,\frac{1}{2} x +x z ,x y },& T_{3},T_{11},  \\  
    \lra{-\frac{1}{2} y +y z ,-\frac{1}{2} x +x z ,x y },& T_{20},T_{23},  \\  
    \lra{-\frac{1}{2} z +y z ,x z ,-\frac{1}{2} x +x y},& T_{2},T_{19},  \\   
    \lra{\frac{1}{2} z +y z ,x z ,\frac{1}{2} x +x y },& T_{12},T_{22},  \\   
    \lra{y z ,-\frac{1}{2} z +x z ,-\frac{1}{2} y +x y },& T_{1},T_{24},  \\ 
    \lra{y z ,\frac{1}{2} z +x z ,\frac{1}{2} y +x y },& T_{4},T_{21},  \\   
    \lra{\frac{1}{2} y +\frac{1}{2}+z +y z ,\frac{1}{2} x +x z ,x y +x },& T_{9},  \\  
    \lra{\frac{1}{2} z +\frac{1}{2}+y +y z ,x z +x ,\frac{1}{2} x +x y },& T_{10},  \\
    \lra{-\frac{1}{2} y +\frac{1}{2}-z +y z ,-\frac{1}{2} x +x z ,x y -x},& T_{15}, 
    \end{cases}} and 
\an{\label{4-p3d2} \begin{cases} 
    \lra{-\frac{1}{2} z +\frac{1}{2}-y +y z ,x z -x ,-\frac{1}{2} x +x y},& T_{16},  \\  
    \lra{y z +y ,\frac{1}{2} z +\frac{1}{2}+x +x z ,\frac{1}{2} y +x y},& T_{5},  \\  
    \lra{\frac{1}{2} y +y z ,\frac{1}{2} x +\frac{1}{2}+z +x z ,x y +y},& T_{6},  \\   
    \lra{y z -y ,-\frac{1}{2} z +\frac{1}{2}-x +x z ,-\frac{1}{2} y +x y},& T_{17},  \\  
    \lra{-\frac{1}{2} y +y z ,-\frac{1}{2} x +\frac{1}{2}-z +x z ,x y -y},& T_{18},  \\ 
    \lra{\frac{1}{2} z +y z ,x z +z ,\frac{1}{2} x +\frac{1}{2}+y +x y},& T_{7},  \\  
    \lra{y z +z ,\frac{1}{2} z +x z ,\frac{1}{2} y +\frac{1}{2}+x +x y},& T_{8},  \\    
    \lra{-\frac{1}{2} z +y z ,x z -z ,-\frac{1}{2} x +\frac{1}{2}-y +x y},& T_{13},  \\  
    \lra{y z -z ,-\frac{1}{2} z +x z ,-\frac{1}{2} y +\frac{1}{2}-x +x y },& T_{14}.
    \end{cases} }

By $\nabla \phi_4$ above, \eqref{4-p3d} and  \eqref{4-p3d2}, 
    we find the integral on the 24 tetrahedra $T_i$ of $R_4$,  
    $(\nabla(p-I_2p),\nabla \phi_4)_{T_i}=$
\a{  \begin{cases}  
\ \ \frac{1}{90}, & i=12  ,\\
-\frac{1}{90}, & i=19 ,\\
\ \ \frac{1}{120}, & i=11,13,18,\\
-\frac{1}{120}, & i=6,7,20 ,\\
\ \ \frac{1}{180}, & i=21 ,\\
-\frac{1}{180}, & i=1 ,\\ \quad & \end{cases} &&  \begin{cases} 
\ \ \frac{1}{360}, & i=22 ,\\
-\frac{1}{360}, & i=2 ,\\
\ \ \frac{1}{240}, & i=3,14,15  ,\\ 
-\frac{1}{240}, & i=8,9,23 ,\\
\ \ \frac{7}{720}, & i=4 ,\\
-\frac{7}{720}, & i=24,\\
0,&i=5,10,16,17    ,
    \end{cases} }
and $    (\nabla(p-I_2p),\nabla \phi_4)_{R_4} = 0$. 
The lemma is proved.      
\end{proof}

\section{Superconvergence}

We prove the $H^1$ and $L^2$ superconvergence in this section.

\begin{theorem}
Let $u \in H^4(\Omega) \cap H^1_0 (\Omega)$
    be the solution of the Poisson equation \eqref{e1-1}--\eqref{e1-2}. 
Let $u\in V_h $ of \eqref{V-h}
    be the solution of the finite element equation   \eqref{finite}.
It holds that   
\an{  \label{sup-1}  |I_h u -u_{h}|_{1} & \le  C h^4  |u|_{4}, }
where $I_h$ is the $P_2$ Lagrange interpolation operator. 
\end{theorem}
\begin{proof}  
 Multiplying \eqref{e1-1} by $  v_h\in  V_h$  and
  doing integration by parts,  we get
\an{\label{p-1}  
     ( \nabla  u,\nabla  v_h) &=(  f, v_h).
} 
 Subtracting \eqref{p-1} from \eqref{finite}, we get
\an{\label{p-2}  
     ( \nabla(  u- u_h), \nabla v_h )=0 \quad \forall v_h \in V_h .
} 

 By \eqref{p-2}, \eqref{d-3-1}, \eqref{d-3-2}, \eqref{d-3-3} and
   \eqref{d-3-4},we have, denoting $v_h=I_h u - u_h\in V_h$,
\a{  |I_h u-u_h|_1^2 &= (\nabla(I_h u-u), \nabla v_h) \\
 &=\quad \sum_{R_1\subset\Omega} c_{1,i}(\nabla(I-I_h)(I-P_{3,1}) u, \nabla \phi_1)_{R_1} \\
&\quad \ + \sum_{R_2\subset\Omega} c_{2,i}(\nabla(I-I_h)(I-P_{3,2}) u, \nabla\phi_2)_{R_2} \\
&\quad \ + \sum_{R_3\subset\Omega} c_{3,i}(\nabla(I-I_h)(I-P_{3,3}) u, \nabla\phi_3)_{R_3} \\
&\quad \ + \sum_{R_4\subset\Omega} c_{4,i}(\nabla(I-I_h)(I-P_{3,4}) u, \nabla\phi_4)_{R_4} \\
    &\le C \bigg(\sum_{j=1}^4 \Big(\sum_{R_j \subset\Omega} |(I-P_{3,j}) u|_{1,R_j}^2\Big)^{\frac 12}
              \bigg) |v_h|_1 \\
    &\le C \bigg(\sum_{j=1}^4 \Big(\sum_{R_j\subset\Omega } h^8|u|_{4,R_j}^2\Big)^{\frac 12}
              \bigg) |v_h|_1 \\
    & = C h^4 |u|_4 |v_h|_1,  }
where $P_{3,j}$ is the local $L^2$ projection operator on to $P_3(R_j)$,
   and $I-I_h$ is $H^1$-stable, cf. \cite{Scott-Zhang}.
Hence, canceling one $| v_h|_{1}$ on each side, we get \eqref{sup-1}.
The proof is complete.
\end{proof}

To prove the $L^2$-supconvergence we use the standard duality argument.  Let 
  $ w \in H^1_0(\Omega)$  such that
  \an{\label{e-w}   -  \Delta  w &=I_h u-u_h \qquad \hbox{in } \Omega.  } 
We assume the following full-regularity,
\an{\label{regularity} | w|_2 \le C \|I_h u-u_h\|_0.  }

\begin{theorem} 
Let $u \in H^4(\Omega) \cap H^1_0 (\Omega)$
    be the solution of the Poisson equation \eqref{e1-1}--\eqref{e1-2}. 
Let $u\in V_h $ of \eqref{V-h}
    be the solution of the finite element equation   \eqref{finite}.
Assuming the full regularity  \eqref{regularity}, it holds that    
\an{ \label{sup-2}
   \| I_h u - u_{h}\|_{0}  & \le  C h^4  | {u} |_{4} . } 
\end{theorem} 
  
\begin{proof} 
Multiplying \eqref{e-w} by $(I_h u - u_{h})$ and
  doing integration by parts,  we get
\an{\label{j1} \|I_h u - u_{h}\|_0^2  &=  (  \nabla  w , \nabla(I_h u - u_{h})) .
} 
Let $ w_h\in   V_h$  be the $P_2$ finite element solution for problem
    \eqref{e-w}, satisfying 
 \an{\label{j1-2}
    ( \nabla  w_h , \nabla v_h )  &=(I_h u - u_{h},  v_h)  \qquad  \forall  v_h & \in V_h . 
    }  
By \eqref{sup-1}, we can get the quasi-optimal error bound for $w_h$ that
\an{\label{h1} |w-w_h|_1 \le |w-I_h w |_1 + |I_h w-w_h|_1 \le Ch |w|_2. }

Subtracting \eqref{j1-2} from \eqref{j1}, by the quasi-optimal error bound \eqref{h1} and
    the assumed elliptic regularity \eqref{regularity},   we get
\an{\label{j3} \ad{ \|I_h u - u_{h} \|_0^2  &=
     ( \nabla( w- w_h), \nabla(I_h u - u_{h})) \\
           & \le C h |w|_2 h^3 |u|_4 \le Ch^4 |u|_4  \|I_h u - u_{h} \|_0. } } 
After canceling one $\|I_h u - u_{h}\|_0$, \eqref{sup-2} follows \eqref{j3}.
 The proof is complete.
\end{proof}

\section{Numerical tests}\label{s-numerical}

We solve the 3D Poisson equation \eqref{e1-1}--\eqref{e1-2} 
   on a unit cube domain $\Omega=(0,1)^3$. We choose the right hand side function
     $f$ so that we have three exact solutions for the numerical test, 
\an{\label{s-1}   u&= 2^6 x (1-x) y (1-y) (z+1) (2-z) z (1-z), \\
   \label{s-2}  u&= \sin(\pi x) \sin( \pi y) \sin(2 \pi z) ,\\
    \label{s-3}  u&=2^{13} ( x  -  x^2 )^2 ( y - y^2 )^2 (2 z - 1)z(z - 1).
   }    
As required by the theory for the superconvergence,  we use
  the  uniform tetrahedral meshes shown in Figure \ref{grid3}.

\begin{figure}[H]
\begin{center}
 \setlength\unitlength{1.1pt}
    \begin{picture}(320,118)(0,3)
    \put(0,0){\begin{picture}(110,110)(0,0) \put(25,102){Grid 1:}
       \multiput(0,0)(80,0){2}{\line(0,1){80}}  \multiput(0,0)(0,80){2}{\line(1,0){80}}
       \multiput(0,80)(80,0){2}{\line(1,1){20}} \multiput(0,80)(20,20){2}{\line(1,0){80}}
       \multiput(80,0)(0,80){2}{\line(1,1){20}}  \multiput(80,0)(20,20){2}{\line(0,1){80}}
    \put(80,0){\line(-1,1){80}}\put(80,0){\line(1,5){20}}\put(80,80){\line(-3,1){60}}
      \end{picture}}
    \put(110,0){\begin{picture}(110,110)(0,0)\put(25,102){Grid 2:}
       \multiput(0,0)(40,0){3}{\line(0,1){80}}  \multiput(0,0)(0,40){3}{\line(1,0){80}}
       \multiput(0,80)(40,0){3}{\line(1,1){20}} \multiput(0,80)(10,10){3}{\line(1,0){80}}
       \multiput(80,0)(0,40){3}{\line(1,1){20}}  \multiput(80,0)(10,10){3}{\line(0,1){80}}
    \put(80,0){\line(-1,1){80}}\put(80,0){\line(1,5){20}}\put(80,80){\line(-3,1){60}}
       \multiput(40,0)(40,40){2}{\line(-1,1){40}}
        \multiput(80,40)(10,-30){2}{\line(1,5){10}}
        \multiput(40,80)(50,10){2}{\line(-3,1){30}}
      \end{picture}}
    \put(220,0){\begin{picture}(110,110)(0,0) \put(25,102){Grid 3:}
       \multiput(0,0)(20,0){5}{\line(0,1){80}}  \multiput(0,0)(0,20){5}{\line(1,0){80}}
       \multiput(0,80)(20,0){5}{\line(1,1){20}} \multiput(0,80)(5,5){5}{\line(1,0){80}}
       \multiput(80,0)(0,20){5}{\line(1,1){20}}  \multiput(80,0)(5,5){5}{\line(0,1){80}}
    \put(80,0){\line(-1,1){80}}\put(80,0){\line(1,5){20}}\put(80,80){\line(-3,1){60}}
       \multiput(40,0)(40,40){2}{\line(-1,1){40}}
        \multiput(80,40)(10,-30){2}{\line(1,5){10}}
        \multiput(40,80)(50,10){2}{\line(-3,1){30}}

       \multiput(20,0)(60,60){2}{\line(-1,1){20}}   \multiput(60,0)(20,20){2}{\line(-1,1){60}}
        \multiput(80,60)(15,-45){2}{\line(1,5){5}} \multiput(80,20)(5,-15){2}{\line(1,5){15}}
        \multiput(20,80)(75,15){2}{\line(-3,1){15}}\multiput(60,80)(25,5){2}{\line(-3,1){45}}
      \end{picture}}

    \end{picture}
    \end{center}
\vspace{0.5cm}
\caption{ The uniform tetrahedral meshes used in the computation.  }
\label{grid3}
\end{figure}
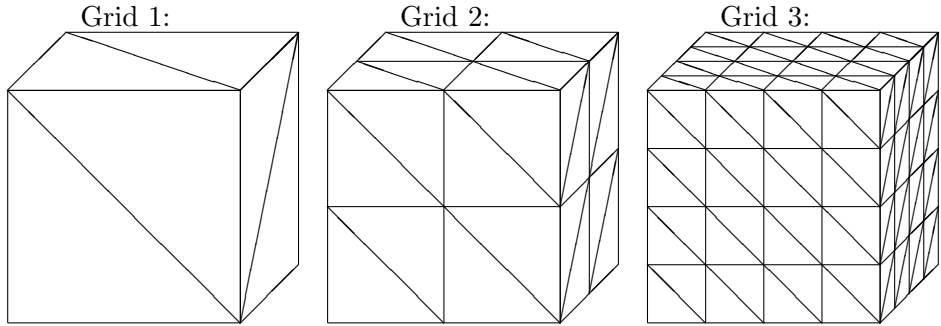

We compute first the solution \eqref{s-1} on uniform tetrahedral grids shown in Figure
   \ref{grid3}, by the $P_2$ Lagrange finite element $V_h$ defined in \eqref{V-h}.  
The results are listed in Table \ref{t-1}, where $L_3$ is the nodal $P_3$
   Lagrange interpolation defined by Figure \ref{p3n} and $L_3 u_h$ is a $P_3$ polynomial
   on each cube.
The first part of Table \ref{t-1} shows the $L^2$ and $H^1$ superconvergence of $u_h$.
The second part of Table \ref{t-1} shows the $L^2$ and $H^1$ optimal-order convergence of $u_h$.
The third part of Table \ref{t-1} shows the $L^2$ and $H^1$ optimal-order convergence of
 the lifted $P_3$ solution, $L_3 u_h$.

\begin{table}[H] 
  \centering  \renewcommand{\arraystretch}{1.1}
  \caption{Computing solutions for \eqref{s-1} on meshes shown in Figure 
      \ref{grid3}. } 
  \label{t-1}
\begin{tabular}{c|cc|cc}
\hline
$G_i$  & $\|I_h u- u_h\|_0 $ & $O(h^r)$ & $|I_h u- u_h|_{1}$ & $O(h^r)$   \\ \hline  
 2 &   0.610E-01 &1.36 &   0.576E+00 &0.82 \\
 3 &   0.604E-02 &3.34 &   0.114E+00 &2.33 \\
 4 &   0.448E-03 &3.75 &   0.175E-01 &2.71 \\
 5 &   0.298E-04 &3.91 &   0.240E-02 &2.87 \\
 6 &   0.190E-05 &3.97 &   0.313E-03 &2.94 \\
 7 &   0.120E-06 &3.99 &   0.399E-04 &2.97 \\ \hline  \hline 
$G_i$  & $\| u- u_h\|_0 $ & $O(h^r)$ & $| u- u_h|_{1}$ & $O(h^r)$   \\ \hline   
 2 &   0.120E+00 &0.00 &   0.149E+01 &0.00 \\
 3 &   0.141E-01 &3.08 &   0.418E+00 &1.84 \\
 4 &   0.165E-02 &3.09 &   0.109E+00 &1.94 \\
 5 &   0.202E-03 &3.03 &   0.276E-01 &1.98 \\
 6 &   0.251E-04 &3.01 &   0.693E-02 &1.99 \\
 7 &   0.313E-05 &3.00 &   0.173E-02 &2.00 \\ \hline  \hline 
$G_i$  & $\| u- L_3 u_h\|_0 $ & $O(h^r)$ & $| u- L_3 u_h|_{1}$ & $O(h^r)$   \\ \hline  
 2 &   0.871E-01 &0.00 &   0.523E+00 &0.00 \\
 3 &   0.208E-01 &2.06 &   0.341E+00 &0.61 \\
 4 &   0.237E-02 &3.14 &   0.748E-01 &2.19 \\
 5 &   0.194E-03 &3.61 &   0.122E-01 &2.62 \\
 6 &   0.138E-04 &3.82 &   0.172E-02 &2.82 \\
 7 &   0.916E-06 &3.91 &   0.228E-03 &2.92 \\
\hline 
    \end{tabular}%
\end{table}%

In Table \ref{t-2}, 
we compute first the solution \eqref{s-2} on uniform tetrahedral meshes shown in Figure
   \ref{grid3}, by the $P_2$ Lagrange finite element $V_h$ defined in \eqref{V-h}.  
The results in Table \ref{t-2} confirm the theory.
The convergence profile is the same as that in Table \ref{t-1}.

\begin{table}[H] 
  \centering  \renewcommand{\arraystretch}{1.1}
  \caption{Computing solutions for \eqref{s-2} on meshes shown in Figure 
      \ref{grid3}. } 
  \label{t-2}
\begin{tabular}{c|cc|cc}
\hline
$G_i$  & $\|I_h u- u_h\|_0 $ & $O(h^r)$ & $|I_h u- u_h|_{1}$ & $O(h^r)$   \\ \hline  
 2 &   0.714E-01 &0.00 &   0.716E+00 &0.00 \\
 3 &   0.103E-01 &2.80 &   0.158E+00 &2.18 \\
 4 &   0.868E-03 &3.56 &   0.263E-01 &2.59 \\
 5 &   0.599E-04 &3.86 &   0.369E-02 &2.83 \\
 6 &   0.386E-05 &3.96 &   0.483E-03 &2.94 \\
 7 &   0.243E-06 &3.99 &   0.614E-04 &2.97 \\ \hline  \hline 
$G_i$  & $\| u- u_h\|_0 $ & $O(h^r)$ & $| u- u_h|_{1}$ & $O(h^r)$   \\ \hline   
 2 &   0.126E+00 &0.00 &   0.156E+01 &0.00 \\
 3 &   0.192E-01 &2.71 &   0.447E+00 &1.81 \\
 4 &   0.216E-02 &3.16 &   0.119E+00 &1.91 \\
 5 &   0.253E-03 &3.09 &   0.305E-01 &1.96 \\
 6 &   0.310E-04 &3.03 &   0.768E-02 &1.99 \\
 7 &   0.385E-05 &3.01 &   0.192E-02 &2.00 \\ \hline  \hline 
$G_i$  & $\| u- L_3 u_h\|_0 $ & $O(h^r)$ & $| u- L_3 u_h|_{1}$ & $O(h^r)$   \\ \hline  
 2 &   0.129E+00 &0.00 &   0.773E+00 &0.00 \\
 3 &   0.228E-01 &2.50 &   0.387E+00 &1.00 \\
 4 &   0.305E-02 &2.90 &   0.989E-01 &1.97 \\
 5 &   0.260E-03 &3.55 &   0.165E-01 &2.58 \\
 6 &   0.183E-04 &3.83 &   0.232E-02 &2.84 \\
 7 &   0.120E-05 &3.93 &   0.303E-03 &2.93 \\
\hline 
    \end{tabular}%
\end{table}%

In Table \ref{t-3}, 
we compute first the solution \eqref{s-3} on uniform tetrahedral meshes shown in Figure
   \ref{grid3}, by the $P_2$ Lagrange finite element $V_h$ defined in \eqref{V-h}.  
The results in Table \ref{t-3} confirm the theory as well.

\begin{table}[ht] 
  \centering  \renewcommand{\arraystretch}{1.1}
  \caption{Computing solutions for \eqref{s-3} on meshes shown in Figure 
      \ref{grid3}. } 
  \label{t-3}
\begin{tabular}{c|cc|cc}
\hline
$G_i$  & $\|I_h u- u_h\|_0 $ & $O(h^r)$ & $|I_h u- u_h|_{1}$ & $O(h^r)$   \\ \hline  
 2 &   0.292E-01 &0.00 &   0.290E+00 &0.00 \\
 3 &   0.444E-02 &2.72 &   0.712E-01 &2.03 \\
 4 &   0.415E-03 &3.42 &   0.132E-01 &2.43 \\
 5 &   0.302E-04 &3.78 &   0.197E-02 &2.74 \\
 6 &   0.199E-05 &3.92 &   0.266E-03 &2.89 \\
 7 &   0.126E-06 &3.97 &   0.345E-04 &2.95 \\ \hline  \hline 
$G_i$  & $\| u- u_h\|_0 $ & $O(h^r)$ & $| u- u_h|_{1}$ & $O(h^r)$   \\ \hline 
 2 &   0.523E-01 &0.00 &   0.641E+00 &0.00 \\
 3 &   0.843E-02 &2.63 &   0.200E+00 &1.68 \\
 4 &   0.996E-03 &3.08 &   0.554E-01 &1.85 \\
 5 &   0.117E-03 &3.09 &   0.144E-01 &1.94 \\
 6 &   0.143E-04 &3.03 &   0.364E-02 &1.98 \\
 7 &   0.178E-05 &3.01 &   0.913E-03 &2.00 \\ \hline  \hline 
$G_i$  & $\| u- L_3 u_h\|_0 $ & $O(h^r)$ & $| u- L_3 u_h|_{1}$ & $O(h^r)$   \\ \hline  
 2 &   0.474E-01 &0.00 &   0.285E+00 &0.00 \\
 3 &   0.109E-01 &2.12 &   0.184E+00 &0.63 \\
 4 &   0.153E-02 &2.83 &   0.498E-01 &1.88 \\
 5 &   0.137E-03 &3.48 &   0.877E-02 &2.50 \\
 6 &   0.101E-04 &3.77 &   0.128E-02 &2.78 \\
 7 &   0.677E-06 &3.90 &   0.171E-03 &2.90 \\
\hline 
    \end{tabular}%
\end{table}%

\end{document}